\documentclass[a4paper,12pt]{article}

\usepackage[T1]{fontenc}
\usepackage{amsmath, graphicx, color, amsthm, centernot, amssymb}
\usepackage[font=small,labelfont=bf,width=12.5cm]{caption}
\usepackage{algorithmic, algorithm,url}
\usepackage{lineno}

\newtheorem{theorem}{Theorem}[section]

\newtheorem{lemma}[theorem]{Lemma}
\newtheorem{corollary}[theorem]{Corollary}
\newtheorem{example}[theorem]{Example}
\newtheorem{conjecture}[theorem]{Conjecture}
\newtheorem{remark}[theorem]{Remark}

\newtheorem{question}[theorem]{Question}

\newcommand{\chis}[1]{\chi_s'(#1)}

\newcommand{\inv}{^{-1}}
\newcommand{\id}{{\mathrm{id}}}

\begin{document}

\title{\textbf{Strong edge colorings of graphs and the covers of Kneser graphs}}

\author{	
	Borut Lu\v{z}ar\thanks{Faculty of Information Studies, Novo mesto, Slovenia.
		E-Mail: \texttt{borut.luzar@gmail.com}}, \
	Edita Ma\v{c}ajov\'{a}\thanks{
		Comenius University, Bratislava, Slovakia.
		E-Mails: \texttt{\{macajova,skoviera\}@dcs.fmph.uniba.sk}}, \
	Martin \v{S}koviera\footnotemark[2],\		
	Roman Sot\'{a}k\thanks{
		Pavol Jozef \v{S}af\'{a}rik University, Ko\v{s}ice, Slovakia.
		E-Mail: \texttt{roman.sotak@upjs.sk}},
}

\maketitle

{%
\abstract{A proper edge coloring of a graph is strong if it
creates no bichromatic path of length three. It is well known
that for a strong edge coloring of a $k$-regular graph at least
$2k-1$ colors are needed. We show that a $k$-regular graph
admits a strong edge coloring with $2k-1$ colors if and only if
it covers the Kneser graph $K(2k-1,k-1)$. In particular, a
cubic graph is strongly $5$-edge-colorable whenever it covers
the Petersen graph. One of the implications of this result is
that a conjecture about strong edge colorings of subcubic
graphs proposed by Faudree et al. [Ars Combin. 29 B (1990),
205--211] is false. }

}%

\section{Introduction}

A \textit{strong edge coloring} of a graph $G$ is a proper edge
coloring with no bichromatic path of length three; in other
words, each color class is an induced matching. The minimum
number of colors for which $G$ admits a strong edge coloring is
called the \textit{strong chromatic index}, and is denoted by
$\chis{G}$. In 1985, Erd\H{o}s and Ne\v{s}et\v{r}il proposed
the following conjecture which was later published
in~\cite{Erd88} and updated by Faudree et
al.~\cite{FauSchGyaTuz90} to fit the graphs with an even or odd
maximum degree.

\begin{conjecture}[Erd\H{o}s, Ne\v{s}et\v{r}il, 1988]\label{conj:Erdos}
The strong chromatic index of an arbitrary graph $G$ satisfies
	$$	
		\arraycolsep=1.4pt\def\arraystretch{1.2}
		\chi_s'(G) \le \left \{
\begin{array}{cl}
	    \tfrac{5}{4} \Delta(G)^2\,, &\quad \textrm{if }
        \Delta(G) \textrm{ is even}\\
		\tfrac{1}{4}(5\Delta(G)^2 - 2\Delta(G) +1)\,, &\quad \textrm{if }
        \Delta(G) \textrm{ is odd.}
\end{array} \right.	
	$$
\end{conjecture}\noindent
Despite many efforts, this conjecture is still widely open and
the best current upper bound $1.772\Delta(G)^2$ (provided that
$\Delta(G)$ is large enough) is due to Hurley et
al.~\cite{HurJoaKan20}.

The motivation for this note comes from one of the two extant
cases of the conjecture of Faudree et
al.~\cite[Section~4]{FauSchGyaTuz90} about strong edge
colorings of subcubic graphs, that is, graphs with maximum
degree $3$.
\begin{conjecture}[Faudree, Schelp, Gy\'{a}rf\'{a}s, Tuza, 1990]\label{conj:main}
Let $G$ be a graph with maximum degree~$3$. Then
\begin{itemize}
\item[$(1)$] $\chis{G} \le 10$,
		
\item[$(2)$] $\chis{G} \le 9$ if $G$ is bipartite,
		
\item[$(3)$] $\chis{G} \le 9$ if $G$ is planar,
		
\item[$(4)$] $\chis{G} \le 6$ if $G$ is bipartite and for
    every edge the sum of degrees of its endvertices is at
    most $5$,
		
\item[$(5)$] $\chis{G} \le 7$ if $G$ is bipartite with
    girth at least $6$,
		
\item[$(6)$] $\chis{G} \le 5$ if $G$ is bipartite with
    large girth.
	\end{itemize}
\end{conjecture}

The first four cases of Conjecture~\ref{conj:main} have already
been resolved. Case~$(1)$, which~is just a special case of the
conjecture of Erd\H{o}s and Ne\v{s}et\v{r}il, has been
confirmed by Andersen~\cite{And92} and by Hor\'{a}k et
al.~\cite{HorQinTro93}. Case~$(2)$ was proved in 1993 by Steger
and Yu~\cite{SteYu93}, and Case~$(3)$, just recently, by
Kostochka et al.~\cite{KosLiRukSanWanYu16}. Case~$(4)$ was
established by Wu and Lin~\cite{WuLin08}; it easily follows
also from a result of Maydanskiy~\cite{May05}. Up to our best
knowledge, Cases~$(5)$ and~$(6)$ are still open, although
several partial results confirming Case~$(6)$ are known
\cite{BorIva13,Lidetal18}.

Our aim is to show that Case~$(6)$ of
Conjecture~\ref{conj:main} is false. In order to be able to
produce infinitely many counterexamples, in
Theorem~\ref{thm:main} we characterize $k$-regular graphs with
strong chromatic index $2k-1$ as those which admit a covering
projection onto the Kneser graph $K(2k-1,k-1)$. In particular,
when $k=3$, a cubic graph is strongly $5$-edge-colorable if and
only if it covers the Kneser graph $K(5,2)$. However, the
latter is nothing but the Petersen graph. Subsequently, in
Theorem~\ref{thm:disproof}, we construct bipartite cubic graphs
of arbitrarily large girth that do not cover the Petersen
graph. By the former result, their strong chromatic index must
be at least $6$.

The last three sections of this paper are devoted to further
aspects of strong $(2k-1)$-colorings of $k$-regular graphs,
with emphasis on the cubic case. In Section~\ref{sec:equiv}, we
reflect on the fact, established in Section~\ref{sec:main},
that the Kneser graph $K(2k-1,k-1)$ has a unique strong
$(2k-1)$-coloring up to automorphism. With the help of a result
borrowed from the theory of voltage graphs we are able to
provide an example of a cubic graph on 40 vertices which covers
the Petersen graph and admits two substantially different
strong $5$-edge-colorings. In Section~\ref{sec:normal}, we
explain the relationship of strong $5$-colorings of cubic
graphs to the famous Petersen coloring conjecture, and in the
final section we present several open problems. Among them, we
propose a strengthening of Case~(5) of
Conjecture~\ref{conj:main}, its only remaining open case.

\section{Main results}\label{sec:main}

Let $\phi$ be a proper edge coloring of a graph $G$. An edge
$e$ of $G$ is said to be \textit{rich} with respect to $\phi$
if all the edges adjacent to $e$ receive pairwise distinct
colors. If $\phi$ is a strong edge coloring, then each edge
must obviously be rich, and vice versa. In particular, every
strong edge coloring of a $k$-regular graph requires at least
$2k-1$ colors. It has been shown in
\cite[Theorem~8]{FauSchGyaTuz90} that this minimum is also
sufficient if $G$ is the Kneser graph $K(2k-1,k-1)$. Recall
that the \textit{Kneser graph} $K(m,n)$, with $m\ge 2n+1$ and
$n\ge 2$, is the graph whose vertices are the $n$-element
subsets of a ground set of $m$ elements, say $\{1,2,\ldots,
m\}$, and where two vertices are adjacent if and only if the
two corresponding sets are disjoint. The Kneser graphs
$K(2k-1,k-1)$ are commonly known as the \textit{odd graphs}
$O_k$ and  have been subject to numerous investigations (see
for example the work of Biggs \cite{Biggs}). The smallest odd
graph $K(5,2)$ is isomorphic to the Petersen graph.

Every odd graph $K(2k-1,k-1)$ has a natural -- or
\textit{canonical} -- strong $(2k-1)$-edge-coloring, which we
denote by $\sigma_k$. It can be described as follows: for any
edge $uv$ of $K(2k-1,k-1)$ the set $u\cup
v\subseteq\{1,2,\ldots, 2k-1\}$ contains precisely $2k-2$
elements, so we can set $\sigma_k(uv)$ to be the missing
element of the ground set. It is easy to see that this coloring
is indeed strong. The canonical strong $5$-edge-coloring
$\sigma_3$ of the Petersen graph is represented in
Figure~\ref{fig:pet}.

Observe that every strong $(2k-1)$-edge-coloring $\sigma$ of
any $k$-regular graph $G$ induces a vertex coloring $\sigma'$
of $G$ where every vertex $v$ is colored with the
$(k-1)$-element set of colors that do not occur on the edges
incident with $v$. We call $\sigma'$ the \textit{derived vertex
coloring}. Since $\sigma$ is strong, the colors of any two
adjacent vertices of $G$ are disjoint $(k-1)$-subsets; in
particular, $\sigma'$ is a proper vertex coloring. For the
Petersen graph the derived coloring $\sigma_3'$ is again
indicated in Figure~\ref{fig:pet}.

It is quite remarkable that the edge coloring $\sigma$ can be
uniquely reconstructed from the vertex coloring $\sigma'$: the
edge $uv$ is colored with the element of the ground set not
occurring in the set $\sigma'(u)\cup\sigma'(v)$. This fact
readily implies that the canonical coloring is a unique strong
$(2k-1)$-edge-coloring of $K(2k-1,k-1)$ up to automorphism of
$K(2k-1,k-1)$. Indeed, consider an arbitrary strong
$(2k-1)$-edge-coloring $\tau$ of $K(2k-1,k-1)$. The derived
vertex coloring $\tau'$ associates with each vertex $v$ of
$K(2k-1,k-1)$ -- which is a $(k-1)$-element subset of
$\{1,2,\ldots,2k-1\}$ -- another $(k-1)$-element subset
$\tau'(v)$ of the same set. In other words, $\tau'$ sends a
vertex of $K(2k-1,k-1)$ to another such vertex. Since $\tau$ is
strong, the assignment $v\mapsto \tau'(v)$ is
adjacency-preserving and injective on the neighbors of $v$. It
means that $\tau'$ determines a degree-preserving endomorphism
$\alpha$ of $K(2k-1,k-1)$, which necessarily must be an
automorphism. The way how the canonical $(2k-1)$-edge-coloring
$\sigma_k$ was defined implies that $\alpha$ transforms
$\sigma_k'$ to $\tau'$, and consequently $\sigma_k$ to $\tau$.
Summing up, any two strong $(2k-1)$-edge-colorings of
$K(2k-1,k-1)$ are equivalent under the action of its
automorphism group. Taking into account the fact that every
automorphism of $K(2k-1,k-1)$ is induced by a permutation of
the ground set \cite[Statement~3.1]{Biggs} we can conclude that
any two strong $(2k-1)$-colorings of $K(2k-1,k-1)$ can be
obtained from each other by a permutation of colors.

\begin{figure}[ht]
\begin{center}	
		\includegraphics{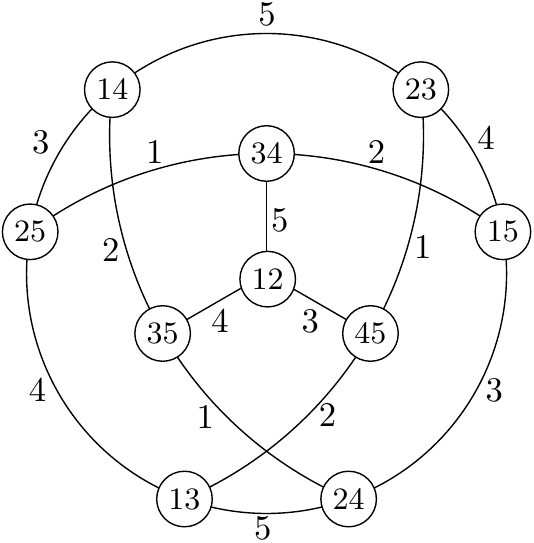}
	\caption{The unique strong $5$-edge-coloring of the Petersen graph
             along with the derived proper vertex coloring}
	\label{fig:pet}
\end{center}
\end{figure}

We aim to prove that all $k$-regular graphs whose strong
chromatic index equals $2k-1$ are closely related to the Kneser
graph $K(2k-1,k-1)$, the relationship being a covering
projection onto $K(2k-1,k-1)$. The pertinent definitions are
now in order.

A surjective graph homomorphism $f\colon \tilde G\to G$ is
called a \textit{covering projection} if for every vertex
$\tilde v$ of $\tilde G$ the set of edges incident with $\tilde
v$ is bijectively mapped onto the set of edges incident with
$f(\tilde v)$. (If $G$ is permitted to contain loops, then the
definition has to be applied to the half-edges incident with
$v$ rather than the edges themselves.) The graph $G$ is usually
referred to as the \textit{base graph} and $\tilde G$ as a
\textit{covering graph} or a \textit{lift} of $G$. A graph
$\tilde G$ \textit{covers} $G$ if there exists such a covering
projection.

It is well known (see \cite{GrossTucker}, Section 2.3) that for
every covering projection $f\colon\tilde G\to G$ there exists a
positive integer $d$ such that every vertex of $G$ has exactly
$d$ preimages and every edge of $G$ has exactly $d$ preimages;
such a cover is said to be \textit{$d$-fold}. For example, the
Petersen graph admits a $5$-fold covering projection onto the
dumbbell graph (which consists of two adjacent vertices and a
loop attached to each of them).

Covering graphs have been useful in numerous parts of graph
theory, especially when a locally defined structure on the base
graph can be `lifted' to the covering graph. This is true, for
example, for edge colorings, flows, embeddings on surfaces, and
other graph structures. The following easy fact is in a similar
vein.

\begin{lemma}\label{lm:lift}
Let $f \colon \tilde{G}\to G$ be a covering projection of
graphs. If $G$ is strongly $n$-edge-colorable for
some integer $n$, then so is $\tilde G$.
\end{lemma}

\begin{proof}
Let $\phi$ be a strong $n$-edge-coloring of $G$. Define an edge
coloring $\tilde\phi$ of $\tilde G$ by setting
$\tilde\phi(x)=\phi(f(x))$ for each edge $x$ of $\tilde G$. As
previously mentioned, every edge of $G$ with respect to $\phi$
is rich. The definition of a covering immediately implies that
the same holds for each edge of $\tilde G$ with respect to
$\tilde\phi$. Therefore $\tilde\phi$ is a strong
$n$-edge-coloring of $\tilde G$.
\end{proof}

Now we are ready for our main results.

\begin{theorem}\label{thm:main}
The strong chromatic index of a $k$-regular graph $G$ equals
$2k-1$ if and only if $G$ covers the Kneser graph
$K(2k-1,k-1)$.
\end{theorem}

\begin{proof}
The backward implication is a direct consequence of
Lemma~\ref{lm:lift}, so we are left with the forward
implication.

Let $\sigma$ be a strong edge coloring of a $k$-regular graph
$G$ with $2k-1$ colors from the set $\{1,2,\ldots, 2k-1\}$.
Without loss of generality we may assume that $G$ is connected.
To define a covering projection $f\colon G\to K(2k-1,k-1)$ we
use the derived vertex coloring $\sigma'$ of $G$. Recall that
for each vertex $v$ of $G$ the color $\sigma'(v)$ is a
$(k-1)$-subset of $\{1,2,\ldots, 2k-1\}$. Thus there is a
unique vertex $\bar v$ of $K(2k-1,k-1)$ such that
$\sigma'(v)=\bar v$. Define $f$ by sending $v$ to $\bar v$. The
mapping is clearly correctly defined.

We first observe that $f\colon G\to K(2k-1,k-1)$ is a
homomorphism. To see this, note that the colors of adjacent
vertices in $G$ are disjoint $(k-1)$-elements sets. It follows
that $f$ sends adjacent vertices $u$ and $v$ to disjoint sets
$\bar u$ and $\bar v$. However, in $K(2k-1,k-1)$ such vertices
are adjacent. Therefore $f$ sends adjacent vertices to adjacent
vertices.

To show that $f$ is a covering projection we need to check that
$f$ takes the neighborhood of every vertex bijectively to
$K(2k-1,k-1)$, and that $f$ is surjective. Consider an
arbitrary vertex $v$ of $G$, and note that the $k$ neighbors
$u_i$ of $v$, where $1 \le i \le k$, receive from $\sigma'$
pairwise distinct colors $\sigma'(u_i)$. Since $f$ sends each
$u_i$ to the vertex $\bar u_i=\sigma'(u_i)$ in $K(2k-1,k-1)$,
it takes the $k$ neighbors of $v$ to $k$ distinct neighbors of
$f(v)$, as required.

Finally, to check that $f$ is surjective it is sufficient to
realize that $f(G)$ is a connected subgraph of $K(2k-1,k-1)$
and that $f(G)$ is $k$-regular. Therefore $f(G)=K(2k-1,k-1)$,
which proves that $f$ is a covering projection.
\end{proof}

For cubic graph the previous theorem amounts to the following.

\begin{corollary}\label{cor:main}
The strong chromatic index of a cubic graph $G$ equals $5$ if
and only if $G$ covers the Petersen graph.
\end{corollary}

Our next theorem refutes Case~$(6)$ of
Conjecture~\ref{conj:main}.

\begin{theorem}\label{thm:disproof}
There exist bipartite cubic graphs with arbitrarily large girth
and strong chromatic index at least $6$.
\end{theorem}

\begin{proof}
We construct an infinite sequence $(G_n)_{n\ge 1}$ of connected
bipartite cubic graphs such that $G_n$ has girth $2^n$ and
order a power of $2$. Since the order of any graph that covers
the Petersen graph is a multiple of $10$, it follows that $G_n$
does not cover the Petersen graph for any $n\ge 1$. From
Corollary~\ref{cor:main} we get that $\chis{G_n}\ge 6$ for each
$n\ge 2$.

We now construct the sequence by induction on $n$. For the
starting graph $G_1$ we take the cubic graph consisting of two
vertices and three parallel edges, which is connected,
bipartite, and has girth $2$. Assume that we have already
constructed the graph $G_n$ for some $n\ge 1$. By the induction
hypothesis, $G_n$ is connected and bipartite with girth $2^n$
and order a power of $2$. To create $G_{n+1}$, we employ the
construction of Exoo and Jajcay described in the proof of
Theorem~3.1 of \cite{ExoJaj11}. Their method uses a
$\mathbb{Z}_2$-homology voltage assignment on a connected graph
$H$ of girth $g$ to produce a $2^{\beta(H)}$-fold covering
projection $\tilde H\to H$ with $\tilde H$ connected of girth
$2g$, where $\beta(G)$ denotes the cycle rank (Betti number)
of~$G$. (We refer the reader for details to \cite{ExoJaj11}.)
If we apply this construction to $G_n$, we obtain a connected
graph $G_{n+1}$ of girth $2^{n+1}$ and order
$2^{\beta(G_n)}m_n$, where $m_n$ is the order of $G_n$. Since
$m_n$ is a power of $2$, so is the order of $G_{n+1}$.
Furthermore, $G_{n+1}$ is bipartite because any covering graph
over a bipartite base graph is bipartite. This concludes the
construction of $(G_n)_{n\ge1}$ as well as the entire proof.
\end{proof}

\begin{remark}
{\rm The reader can check that the graph $G_2$ constructed in
the previous proof is isomorphic to the graph of the $3$-cube
$Q_3$, whose strong chromatic index equals $6$. Since the
composition of covering projections is again a covering
projection, each $G_n$ with $n\ge 3$ covers $G_2$, and
therefore it is strongly $6$-edge-colorable by
Lemma~\ref{lm:lift}. Theorem~\ref{thm:main} now implies that
$\chis{G_n}=6$ for each $n\ge 2$. }
\end{remark}

\section{Equivalence of coverings and colorings}\label{sec:equiv}

We have proved that the odd graph $K(2k-1,k-1)$ admits a unique
strong $(2k-1)$-edge-coloring up to automorphism. It is
therefore natural to ask whether the same holds for the graphs
that cover it. As we shall see in this section, the answer is
negative, which at the first glance might seem to be rather
counter-intuitive.

We first show that the problem of finding two essentially
different strong $(2k-1)$-edge-colorings of a $k$-regular graph
is closely related to the problem of finding two non-equivalent
covering projections of the same $k$-regular graph onto the odd
graph $K(2k-1,\penalty0 k-1)$.

We call two edge colorings $\phi_1$ and $\phi_2$ of a
graph $G$ \textit{equivalent} if there exists an automorphism
$\alpha$ of $G$ such that $\phi_2=\phi_1\alpha$. Similarly, we
say that two covering projections $f_1\colon G_1\to G$ and
$f_2\colon G_2\to G$ are \textit{equivalent} if there exists an
isomorphism $\xi\colon G_2\to G_1$ such that $f_2= f_1\alpha$.

The following theorem shows that for strong
$(2k-1)$-edge-colorings of $k$-regular graphs equivalence of
colorings coincides with equivalence of coverings.

\begin{theorem}\label{thm:equivcol}
Every strong $(2k-1)$-edge-coloring $\sigma$ of a $k$-regular
graph $G$ determines a unique covering projection
$f_{\sigma}\colon G\to K(2k-1,k-1)$. Moreover, two such
colorings $\sigma$ and $\tau$ are equivalent if and only if the
corresponding covering projections $f_{\sigma}$ and
$f_{\tau}$ of $G$ are equivalent.
\end{theorem}

\begin{proof}
Recall that given a strong $(2k-1)$-edge-coloring $\sigma$ of
an arbitrary graph $G$ we have defined a covering projection
$f$ by sending an arbitrary vertex $v$ of $G$ to the vertex
$\sigma'(v)$, where $\sigma'$ is the derived vertex coloring of
$G$ with colors being the $(k-1)$-element subsets of the
$(2k-1)$-element set. This is the required covering projection
$f_{\sigma}$ corresponding to the coloring~$\sigma$.

Assume that $\sigma$ and $\tau$ are equivalent strong
$(2k-1)$-edge-colorings of $G$, and let $\alpha$ be an
automorphism of $G$ such that $\tau=\sigma\alpha$. It follows
that $\tau'=\sigma'\alpha$ and therefore, by the definition of
the covering projection corresponding to a strong
$(2k-1)$-edge-coloring, $f_{\tau}=f_{\sigma}\alpha$.
Conversely, if covering projections $f_{\sigma}, f_{\tau}\colon
G\to K(2k-1,k-1)$ are equivalent, then there exists an
automorphism $\beta$ of $G$ such that
$f_{\tau}=f_{\sigma}\beta$. The latter can be rewritten as
$\tau'=\sigma'\beta$. However, the way how the original strong
$(2k-1)$-edge-coloring can be reconstructed from the derived
vertex coloring implies that $\tau=\sigma\beta$, which means
that the colorings $\sigma$ and $\tau$ are equivalent.
\end{proof}

What remains to be done is to find a $k$-regular graph $G$ that
admits two non-equivalent covering projections on the Kneser
graph $K(2k-1,k-1)$. We do it for $k=3$, in which case
$K(2k-1,k-1)$ coincides with the Petersen graph. For this
purpose we need to recall several notions pertaining to the
theory of graph covers.

First of all, it will be convenient to regard each edge
(including the loops) as a pair of oppositely oriented
\textit{darts}. Each dart $x$ directed from $u$ to $v$ has its
unique \textit{inverse} $x\inv$ directed from $v$ to $u$. The
set of all darts of a graph $K$ is denoted by $D(K)$. The
\textit{symmetric group} of all permutations of the $d$-element
set $\{1,2,\ldots,d\}$ is denoted by $S_d$; it acts on
$\{1,2,\ldots,d\}$ on the right.

A \textit{permutation voltage assignment} on a graph $K$ is a
mapping $\kappa\colon D(K)\to S_d$ such that
$\kappa(x\inv)=\kappa(x)\inv$ for each dart $x\in D(K)$. For
convenience, we often denote $\kappa(x)$ by $\kappa_x$. Every
permutation voltage assignment $\kappa$ on $K$ gives rise to
the \textit{derived graph} $K^{\kappa}$ for $K$, or the
\textit{lift} of $K$, which is defined as follows. Set
$V(K^{\kappa})=V(K)\times\{1,2,\ldots,d\}$,
$D(K^{\kappa})=D(K)\times\{1,2,\ldots,d\}$, and for each dart
$x=uv$ and $i\in\{1,2,\ldots,d\}$ let the lifted dart $(x,i)$
join $(u,i)$ to $(v,(i)\kappa_x)$. It is easy to see that the
natural projection $p_{\kappa}\colon K^{\kappa}\to K$ which
erases the second coordinate is a covering projection.
Moreover, a classical result of the theory of voltage graphs
states that every $d$-fold covering projection $\tilde K\to K$
is equivalent to the natural projection $K^{\kappa}\to K$ for a
suitable permutation voltage assignment $\kappa$ on $K$ with
values in $S_d$, see \cite[Theorem~2.4.5]{GrossTucker}.

Let us henceforth assume that the base graph $K$ is connected.
Pick a spanning tree $T$ of $K$ and let $r$ be an arbitrary
vertex of $K$, the \textit{root}. For any vertex $w$ of $K$ let
$T(w)$ denote the unique directed path from the root to $w$,
encoded as a sequence of darts, and let $T(w)\inv$ be the
inverse path. Further, for any dart $z=uv$ whose underlying
edge is not contained in $T$ define the permutation
$\theta(z)\in S_d$ by taking the product of voltages (that is,
values of $\kappa$) on the closed walk $T(u)zT(v)\inv$, rooted
at $r$, in the order determined by the walk and starting from
the root. If we set $\kappa'(x)=\theta(x)$ whenever $x$ is a
cotree dart and $\kappa'(x)=\id$ otherwise, we obtain a new
voltage assignment $\kappa'$ on $K$ called the
\textit{$(T,r)$-reduction} of $\kappa$.

The concept of a $(T,r)$-reduction of a voltage assignment is
quite useful. For example, it can be shown that the derived
graph $K^{\kappa}$ is connected if and only if the voltages of
$\kappa'$ generate a transitive subgroup of $S_d$. What is more
important for us, it can be used to determine whether or not
two covering projections are equivalent. The respective result
is taken from \cite[Theorem~2]{Sko}.

\begin{theorem}\label{thm:equivcp}
Let $\kappa$ and $\lambda$ be permutation voltage assignments
on a connected graph $K$, both having their values in the
symmetric group $S_d$, and let $\kappa'$ and $\lambda'$ be
their $(T,r)$-reductions. The natural projections $p_{\kappa}$
and $p_{\lambda}$ are equivalent if and only if there exists an
inner automorphism $\gamma$ of $S_d$ such that
$\lambda'=\gamma\kappa'$.
\end{theorem}

Now we are prepared to describe an example of a cubic graph
that covers the Petersen graph and has two non-equivalent
strong $5$-edge-colorings. It is depicted in
Figure~\ref{fig:lift}.

\begin{example}\label{ex:non-equiv}{\rm
Consider the permutation voltage assignments $\kappa$ and
$\lambda$ on the Petersen graph with values in $S_4$ which are
represented in Figure~\ref{fig:pva}; the edges not labelled
carry the trivial voltage $\id$ (in both directions). The
values attached to all the edges are involutions, which means
that they unambiguously represent the respective voltage
assignments.
\begin{figure}[htp!]
	$$
		\includegraphics{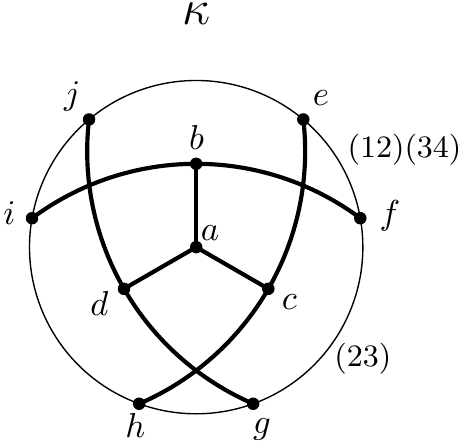} \quad \quad \quad
		\includegraphics{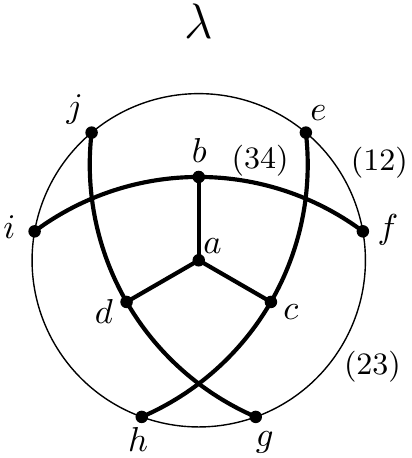}
	$$
\caption{Permutation voltage assignments $\kappa$ and
$\lambda$ on the Petersen graph with values in $S_4$.}
\label{fig:pva}
\end{figure}
The corresponding lifts are isomorphic graphs as can be easily
detected from Figure~\ref{fig:lift}; for sim\-plicity, a vertex
$(v,i)$ is denoted by $v_i$.
\begin{figure}[htp!]
	$$
		\includegraphics{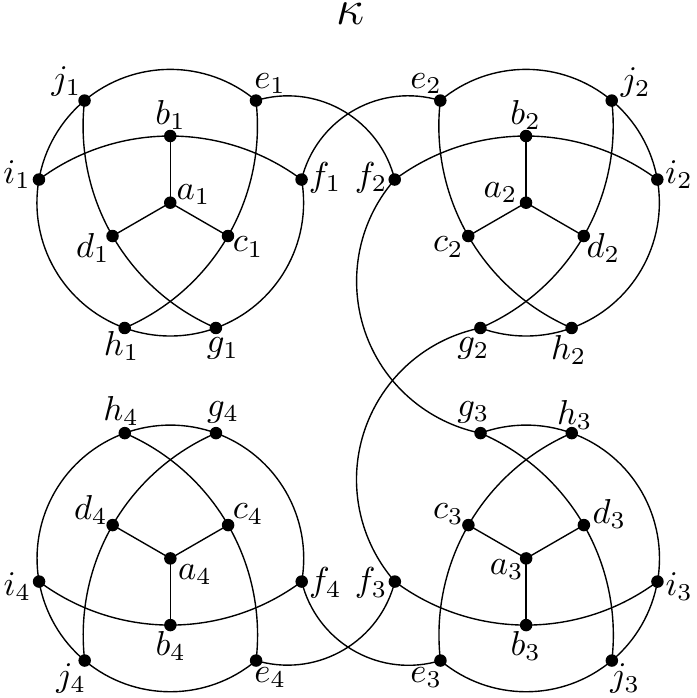} \quad \quad
		\includegraphics{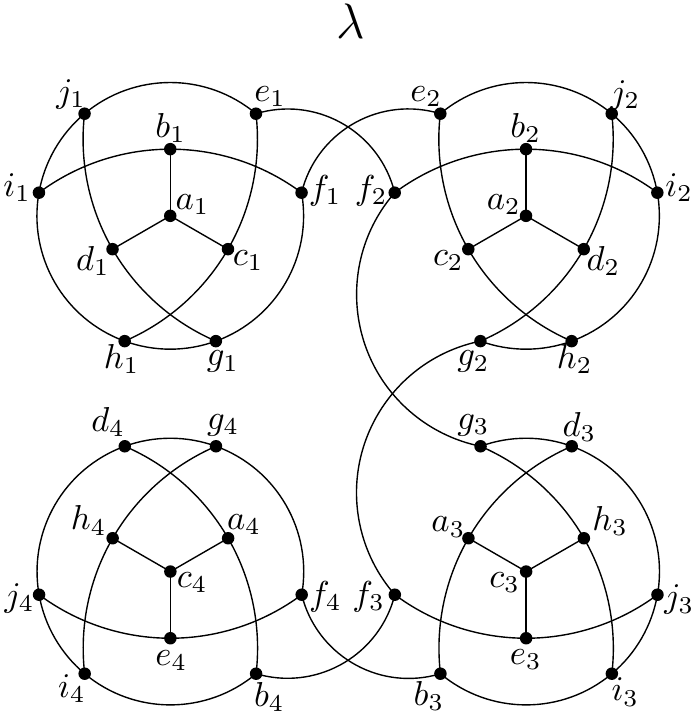}
	$$
\caption{The lifts of the Petersen graph corresponding to the
permutation voltage assignments $\kappa$ and $\lambda$.}
\label{fig:lift}
\end{figure}
We claim that
the natural projections $p_{\kappa}\colon P^{\kappa}\to P$ and
$p_{\lambda}\colon P^{\lambda}\to P$, where $P$ denotes the
Petersen graph, are not equivalent. To see it, pick the
spanning tree $T$ indicated in Figure~\ref{fig:pva} by bold
lines and choose the central vertex $a$ to be the root.
Clearly, the $(T,a)$-reduction $\kappa'$ of $\kappa$ coincides
with $\kappa$. The $(T,a)$-reduction $\lambda'$ of $\lambda$ is
shown in Figure~\ref{fig:reduced}; the value $\lambda'(gf)$ is
not an involution and holds for the indicated direction.
\begin{figure}[htp!]
	$$
		\includegraphics{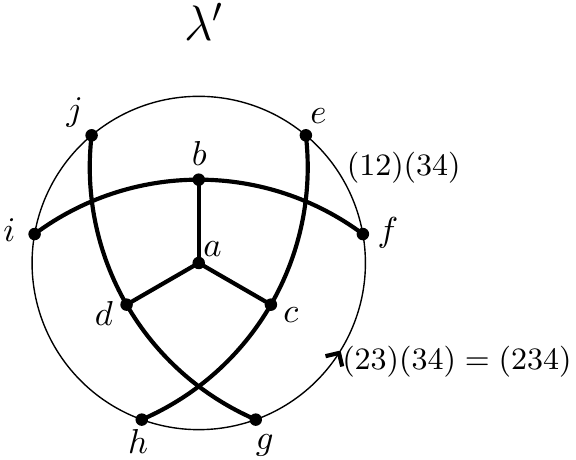}
	$$
\caption{The $(T,a)$-reduction $\lambda'$ of $\lambda$.}
\label{fig:reduced}
\end{figure}
By comparing $\kappa'$ and $\lambda'$ we immediately conclude
that no inner automorphism $\gamma$ of $S_d$ such that
$\lambda'=\gamma\kappa'$ can exist, because every inner
automorphism preserves the cycle structure of permutations.
Theorem~\ref{thm:equivcp} now implies that the covering
projections $p_{\kappa}$ and $p_{\lambda}$ are not equivalent.
Further, from Theorem~\ref{thm:equivcol} we conclude that there
exist strong $5$-edge-colorings $\sigma$ and $\tau$ of the
graph shown in Figure~\ref{fig:lift} such that $f_{\sigma}$ is
equivalent to $p_{\kappa}$ and $f_{\tau}$ is equivalent to
$p_{\lambda}$. They can be constructed simply by lifting the
strong $5$-edge-coloring $\sigma_3$ of the Petersen graph via
$p_{\kappa}$ and $p_{\lambda}$, respectively. By the same
theorem, the colorings $\sigma$ and $\tau$ are not equivalent.
}%
\end{example}

\section{Strong, normal, and Petersen colorings}\label{sec:normal}

Corollary~\ref{cor:main} links strong edge colorings of cubic
graphs to two other interesting types of edge colorings of
cubic graphs -- normal colorings and Petersen colorings -- and
through them to the outstanding Petersen coloring conjecture of
Jaeger \cite{Jae88}.

\begin{conjecture}[Jaeger, 1988]\label{conj:peter}
Every bridgeless cubic graph admits a Petersen coloring.
\end{conjecture}

For a cubic graph $G$ a mapping $\xi\colon E(G)\to E(P)$ is
said to be a \textit{Petersen coloring} if any two adjacent
edges of $G$ are mapped to adjacent edges of the Petersen
graph. As a consequence, for every vertex $v$ of $G$ the three
edges incident with $v$ are mapped to three edges incident with
a vertex of $P$ (as $P$ is triangle-free); in particular, $\xi$
is a proper edge coloring. Nevertheless, a Petersen coloring
need not be a homomorphism $G\to P$ because the induced mapping
between the vertex sets need not send adjacent vertices of $G$
to adjacent vertices of $P$ (for example, it can send them to
the same vertex). If $\xi$ does map adjacent vertices to
adjacent ones, then it determines a covering projection $G\to
P$. Conversely, every covering projection $G\to P$ gives rise
to a Petersen coloring of $G$.

In \cite[Section~5]{Jae85} Jaeger proved that a cubic graph
admits a Petersen coloring if and only if it has a
$5$-edge-coloring which he termed `normal'. A proper
$5$-edge-coloring $\phi$ of a cubic graph $G$ is said to be
\textit{normal} if for every edge $e$ of $G$ the number of
colors on the edges adjacent to $e$ is either $2$ or $4$, but
never $3$. In the former case, $e$ is called \textit{poor}, as
opposed to \textit{rich} (introduced in
Section~\ref{sec:main}), which corresponds to the latter case.
Clearly, a normal $5$-edge-coloring with no poor edges is
strong.

By combining these observations with Corollary~\ref{cor:main}
we obtain the following four equivalent statements.

\begin{theorem}\label{thm:equiv}
Let $G$ be cubic graph. Then the following statements are
equivalent:
\begin{itemize}
\item[$(i)$] $\chi_s'(G) = 5$;
\item[$(ii)$] $G$ covers the Petersen graph;
\item[$(iii)$] $G$ admits a normal $5$-edge-coloring in which
    every edge is rich;
\item[$(iv)$] $G$ admits a Petersen coloring which is a
    homomorphism.
\end{itemize}
\end{theorem}

\section{Conclusion}\label{sec:conclusion}

Coloring subcubic graphs with five colors without creating a
bichromatic path of length $3$ is very restrictive. For
example, we already need five colors when there are two
adjacent vertices of degree $3$ in a graph. Therefore, the
existence of infinitely many cubic graphs admitting a strong
$5$-edge-coloring is quite surprising.

\begin{figure}[htp!]
	$$
		\includegraphics{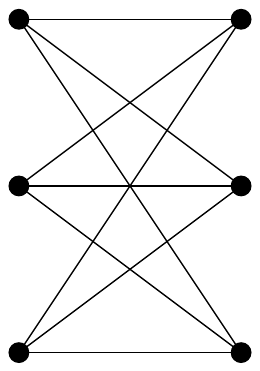} \quad \quad \quad
		\includegraphics{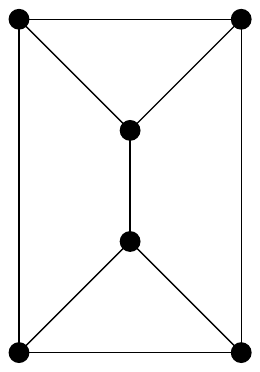} \quad \quad \quad
		\includegraphics{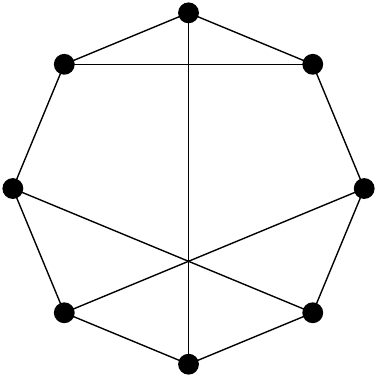}
	$$
	$$
		\includegraphics{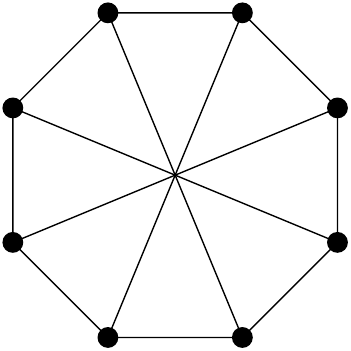} \quad \quad \quad
		\includegraphics{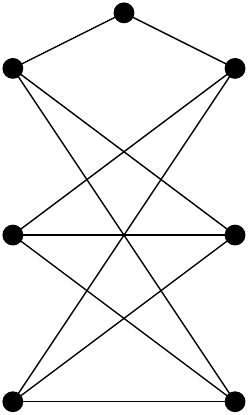}
	$$
\caption{Subcubic graphs with high strong chromatic indices;
the three graphs above need nine colors and the two graphs below
ten colors for any strong edge coloring.}	
\label{fig:9}
\end{figure}

What if we have more available colors? It is already known that
the Heawood graph, a cubic graph on $14$ vertices of girth $6$,
has strong chromatic index $7$. Computational results reveal
that there are two and six cubic graphs of girth at least $6$
on $16$ and $18$ vertices, respectively, for which six colors
are not sufficient. It is even not enough to slightly increase
the girth. The Tutte $8$-cage, a cubic graph of girth $8$ on
$30$ vertices, does not admit any strong $6$-edge-coloring, nor
do five cubic graphs of girth $8$ on 38 vertices. This
encourages us to ask the following.

\begin{question}
Is there a constant $C$ such that every cubic graph of girth at
least $C$ admits a strong $6$-edge-coloring?
\end{question}

Perhaps one can find a characterization of cubic graphs
admitting a strong $6$-edge-coloring similar to the one
presented in this paper for five colors. However, it is likely
that the condition of being bipartite will not play any role in
answering the above question, because Tutte's $8$-cage is
bipartite.

Having seven colors available, it seems that cubic graphs
become strongly colorable as soon as they do not contain short
cycles. Our computational experiments show that all cubic
graphs of girth at least $5$ on at most $26$ vertices and all
bridgeless subcubic graphs of girth at least $5$ on at most
$18$ vertices admit a strong $7$-edge-coloring. We therefore
propose the following conjecture, which strengthens Case
$(5)$ of Conjecture~\ref{conj:main}.

\begin{conjecture}\label{conj:new}
Let $G$ be a subcubic graph of girth at least $5$. Then
	$$
		\chis{G} \le 7\,.
	$$
\end{conjecture}

We remark that without restricting girth eight colors are still
not sufficient to color all bridgeless subcubic graphs (see
Figure~\ref{fig:9} for some examples of graphs that need more
colors). However, as suggested in \cite{HocLajLuz20} on the basis of
computational evidence, it seems that nine colors should
suffice to color all bridgeless subcubic graphs that are not
isomorphic to either $K_{3,3}$ with one subdivided edge or to
the $8$-vertex M\"obius ladder known as the Wagner graph.

\begin{conjecture}[Hocquard, Lajou, Lu\v{z}ar, 2020]
Every bridgeless subcubic graph not isomorphic to the Wagner
graph and the complete bipartite graph $K_{3,3}$ with one edge
subdivided admits a strong $9$-edge-coloring.
\end{conjecture}

In fact, exhaustive computational search indicates that
there are no bridgeless subcubic graphs on more than 12
vertices and strong chromatic index at least $9$ (apart from
the graphs in Figure~\ref{fig:9} and four graphs on 12
vertices). Thus we can strengthen the above conjecture to the
following.

\begin{conjecture}
If $G$ is a bridgeless subcubic graph on at least $13$
vertices, then
	$$
		\chis{G} \le 8\,.
	$$
\end{conjecture}

\paragraph{Acknowledgment.}
The first author was partially supported by the Slovenian
Research Agency Program P1--0383 and the project J1--1692. The
second and the third author received partial support from
APVV--19--0308 and VEGA 1/0813/18. The fourth author was
supported by APVV--19--0153.

\end{document}